\documentclass[a4paper,12pt]{amsart}

\usepackage[latin1]{inputenc}
\usepackage[T1]{fontenc}
\usepackage[english]{babel}

\usepackage{mathrsfs}
\usepackage{amscd}
\usepackage{amsfonts}
\usepackage{amsmath}
\usepackage{amssymb}
\usepackage{amstext}
\usepackage{amsthm}
\usepackage{amsbsy}

\usepackage{xspace}
\usepackage[all]{xy}
\usepackage{graphicx}
\usepackage{url}
\usepackage{latexsym}

\newtheorem{theo}{Theorem}[section]
\newtheorem{lem}[theo]{Lemma}

\newtheorem{ex}[theo]{Example}

\theoremstyle{definition}

\newtheorem{rem}[theo]{Remark}

\newcommand{\Q}{\ensuremath{\mathbb{Q}}}
\newcommand{\R}{\ensuremath{\mathbb{R}}}

\newcommand{\A}{\ensuremath{\aleph}}

\newcommand{\M}{\ensuremath{\mathcal{M}}}

\newcommand{\vs}{\vspace{0.3cm}}

\DeclareMathOperator{\dcl}{dcl}

\newcommand{\book}[2]{{\scshape#1}, {\bf #2}}

\newcommand{\publ}[6]
{{\scshape#1}, #2, {\itshape #3}, {\bf #4} (#5), pp.~#6.}

\title[]{A note on $\aleph_{\alpha}$-saturated o-minimal expansions of real closed fields}

\author{Paola D'Aquino}
\address{Dipartimento di Matematica, Seconda Universit\`a di Napoli, Italy}
\email{paola.daquino@unina2.it}
\author{Salma Kuhlmann}
\address{FB Mathematik \& Statistik, Universit\"at Konstanz,
Germany}
 \email{salma.kuhlmann@uni-konstanz.de}
\date{\today}
\subjclass[2000]{ Primary: 06A05, 12J10, 12J15, 12L12, 13A18;
Secondary: 03C60, 12F05, 12F10, 12F20.} \keywords{power bounded
o-minimal theory, convex valuations, archimedean components, field
of exponents, ordered vector spaces, valuation inequality,
coarsening of a valuation}

\begin{document}
\dedicatory {Dedicated to Professor Yurii Ershov on his 75th
birthday}

\begin{abstract}
We give necessary and sufficient conditions for a polynomially
bounded o-minimal expansion of a real closed field (in a language of
arbitrary cardinality) to be $\aleph_{\alpha}$-saturated. The
conditions are in terms of the value group, residue field, and
pseudo-Cauchy sequences of the natural valuation on the real closed
field. This is achieved by an analysis of types, leading to the
trichotomy.  Our characterization provides a construction method for
saturated models, using fields of generalized power series.

\end{abstract}

 \maketitle


\vs
\section {Introduction}
Let $\mathcal L$ be a language containing $<$, an expansion of a
real closed field $\M = \langle M, +,\cdot , 0,1, <, \dots \rangle$
is o-minimal if every subset of $M$ which is definable with
parameters in $M$ is a finite union of intervals in $M$. For basic
definitions and properties of o-minimal theories see
\cite{speissegger}.

In this paper we are interested in a valuation theoretic
characterization of the $\A_{\alpha}$-saturated models of the
o-minimal theory of $\M\>$. For the notion of Hausdorff's
$\eta_{\alpha}$-sets, see \cite{rosenstein}. Denote by $|A|$ the
cardinality of $A$, and by $|\mathcal L|:=\ell $ that of the
language. Let $\dcl(A)$ denote the definable closure of $A\subset
M$. We have: $\>|\dcl(A)|\> = \> \max \{|A|\>, \ell\}\>$.
O-minimality implies that a formula $\varphi(x\>,\overline{a})$ in
one variable $x$ and parameters $\overline{a}$ is equivalent to a
disjunction of intervals with endpoints in $\dcl(\overline{a})\>$,
thus realizing a type in one variable $x$ and parameters in $A$
amounts to realizing a cut in the $\dcl(A)$. Let $|A| <
\aleph_{\alpha}$. In the case when $\aleph_{\alpha}
> \ell$, then also $|\dcl(A)| < \aleph_{\alpha}$. Therefore in that case $\M$
is $\aleph_{\alpha}$-saturated if and only if it is
$\aleph_{\alpha}$-saturated as a linear order, that is an
$\eta_{\alpha}$-set. So our result gives also a characterization of
$\eta_{\alpha}$-models (for $\alpha$ with $\aleph_{\alpha}
> \ell\>$), extending \cite{allingkuhlmann}.

We assume familiarity with the notions of independence and dimension
in o-minimal structures. We note that $\> \dim ( \dcl(A)) \leq
|A|\>$. For notions and results concerning power bounded o-minimal
expansions needed below see \cite{speissegger}. We shall need the
following fact. An o-minimal theory admits (up to isomorphism) a
unique prime model $\mathcal{P}$ with underlying set $\dcl
(\emptyset) =P$. For example, if $\M$ is a divisible abelian group
then $\dcl(A)$ in the $\mathbb Q$-vector space generated by $A$, the
prime model is $\mathbb{Q}$, and the dimension is the $\Q$-linear
dimension. If $\M$ is a real closed field then $\dcl(A)$, the
relative algebraic closure of the field $\mathbb Q(A)$ in $M$, the
prime model is $\mathbb Q^{rc}$, the field of real algebraic
numbers, and the dimension is the (absolute) transcendence degree.

Our main result in the polynomially bounded case, that is, power
bounded with archimedean field of exponents, is Theorem
\ref{expansions} (see Remark \ref{napm} for the power bounded case).

Note that condition (1) in Theorem \ref{expansions} can be replaced
by the equivalent valuation theoretic characterization of
$\aleph_{\alpha}$- saturated ordered $Q$-vector spaces, which we
provide in Theorem \ref{doag}. Theorem \ref{expansions}
generalizes \cite[Theorem 6.2]{kkmz} which
treated the special case when $\M$ is just a real closed field.

\section{The case of ordered $Q$-vector spaces}
\label{DOAG} Throughout this section we assume $Q$ to be an
Archimedean ordered field, so $Q$ is a subfield of $\mathbb R$. We
recall some general definitions about ordered $Q$-vector spaces, see
\cite{book} for more details. Let $G$ be an ordered $Q$-vector
space, for any $x\in G$ let $|x|=\max \{x,-x\}$.  For non-zero $x, y
\in G$ we say that $x$ is $Q-$archimedean equivalent to $y$  if
there exists $q \in Q$ such that $q|x| \geq |y| $ and $ q|y| \geq
|x|. $ We denote this relation by $\sim_Q$. We write $x<<_Qy$ if
$q|x| < |y|$ for all $q \in Q$. Clearly,  $\sim_Q$ is an equivalence
relation.
 Let $\Gamma : = \{[x] : x \in
G\>,\> x\not=0 \>\}$ the set of equivalence classes. We define a
linear order $<$ on $\Gamma$ as follows, $[y]\, < [x] $ if $x << y$
(notice the reversed order). The valuation associated to $G$ as an
ordered $Q$-vector space is the map $v_Q \colon G\ \longrightarrow
\Gamma \cup \{\infty\} $ defined by  $v_Q(0)= \infty $ and  $v_Q(x)=
[x]$  if $x \neq 0$. It satisfies the following axioms:
$v_Q(x)=\infty$ if and only if $x=0$; $v_Q(qx)=v_Q(x)$ for all
$q\not=0$; and $v_Q(x-y)\geq\min\{v_Q(x)\>, v_Q(y)\>\}$ (and
equality holds for distinct values).
 We call $\Gamma$ the value set of $G$.

We shall need the following valuation inequality for ordered
$Q$-vector spaces. For the discussion below, let us fix a set of
representatives $0\not=g_{\alpha}\in G$ such that
$\gamma_{\alpha}=v_Q(g_{\alpha})$.
 \begin{lem}
The cardinality of the value set $\Gamma$ is less or equal than the
dimension of $G$ over $Q$, i.e. $ |\Gamma | \leq \dim_Q(G)$.
\end{lem}
\begin{proof}
Let $ \Gamma=\{ \gamma_{\alpha}: \alpha <\kappa\}$ and
$|\Gamma|=\kappa$, for a cardinal $\kappa$. We claim that
$\{g_{\alpha}:\alpha <\kappa \}$ are $Q$-independent. If not, there
are $g_1, \ldots ,g_n$ and $q_1, \ldots q_n\in Q^{\times}$ such that
$\sum_{i=1}^nq_ig_i=0$. Now since the values are pairwise distinct,
we have $\infty = v_Q(\sum_{i=1}^nq_ig_i)=\min \{
v_Q(g_i):i=1,\ldots ,n\} $, a contradiction.
\end{proof}
For every $\gamma \in \Gamma$, fix $A_{\gamma}$ a maximal
archimedean $Q$-subspace of $G$ containing $g_{\gamma}$,
the archimedean component associated to $\gamma$.
We recall that, for a limit ordinal $\lambda$, a sequence
$(a_{\rho})_{ \rho < \lambda}$ is pseudo Cauchy if for every $\rho <
\sigma < \tau < \lambda$ we have $ v_Q(a_{\sigma} - a_{\rho})\ <\
v_Q(a_{\tau} - a_{\sigma})\>$,
and $a\in G$ is a pseudo limit if for all $\rho < \lambda\>, \>
v_Q(a - a_{\rho}) = v_Q(a_{\rho + 1} - a_{\rho})\>.$

We now give a characterization of $\A_{\alpha}$-saturation for
ordered $Q$-vector spaces in the language $\mathcal L_{Q}$    of
ordered groups $\mathcal L_{OG}$ expanded with constants for the elements of $Q$, i.e.
$\mathcal L_{Q}=\mathcal L_{OG}\cup \{c_q:q\in Q\}$. This is a
generalization of  the characterization for divisible ordered
abelian groups, see \cite{sgr}.
\begin{theo}
\label{doag} Let $G$ be an ordered $Q$-vector space. Then $G$ is
$\A_{\alpha}$-saturated in the language $\mathcal L_{Q}$ if and only

\begin{enumerate}

\item
its value set $\Gamma $ is an $\eta_{\alpha}$-set


\item
all its Archimedean components are isomorphic to $\R$

\item
every pseudo Cauchy sequence in a $Q$-subspace of $G$ of dimension
$<\A_{\alpha}$ has a pseudo limit in $G$.

\end{enumerate}
\end{theo}
 We  note that the characterization provides a method for
constructing $\A_{\alpha}$-saturated ordered $Q$-vector spaces,
using Hahn group constructions (see \cite{book}).
Let $\Gamma$ be any $\eta_{\alpha}$-set, and let $G=
\Pi_{\Gamma}\mathbb R$ the Hahn product. Then $G$ is
$\A_{\alpha}$-saturated.



\section{The case of power bounded o-minimal expansions}\label{main}
\label{RCF} \medskip\noindent
 If  $R$ is an ordered field
then its natural valuation $v$ has valuation ring  the convex hull
of $\Q$ in $R\>$.  The residue field $k$ is archimedean, i.e.
(isomorphic to) a subfield of $\mathbb R$. We denote the value group
$v(R)$ by $G$. Note that $G$ is divisible and its rational rank is
the linear dimension as a $\mathbb Q$-vector space. See \cite{book}
for details. A characterization of $\A_{\alpha}$-saturated real
closed fields was obtained in \cite[Theorem 6.2]{kkmz}. We note that
in the proof of \cite[Theorem 6.2]{kkmz}, the dimension inequality
(the rational rank of the value group is bounded by the absolute
transcendence degree; see \cite{ershov}) is crucially used. In this
Section, we prove a generalization of this result to o-minimal
expansions of a real closed field $\mathcal M=(M, +,\cdot
,0,1,<,\cdots)\>.$ Recall that the expansion $\mathcal M$ is power
bounded if for each definable function $f:\mathcal M \rightarrow
\mathcal M $ there is $\lambda \in M$ such that $|f(x)|\leq
x^{\lambda}$ for all sufficiently large $x>0$ in $M$. We denote the
field of exponents (of all definable power functions) by $Q$.
\begin{rem}\label{tconvx2}
A power bounded $\mathcal M$ is polynomially bounded if and only if
$Q$ is archimedean. We shall assume that the prime model (of the
theory of $\mathcal M$) is archimedean. In that case, the valuation
ring of the natural valuation is $T$-convex and the value group $G$
is an ordered $Q-$vector space. The following analogue of the
dimension inequality (the valuation inequality) holds: $\dim_Q G\leq
\dim(\mathcal M)$. In particular, $\>\dim_Q(v(\dcl(A)))\leq \dim (
\dcl(A)) \leq |A|.$ See \cite{speissegger} for details.
\end{rem}
\begin{theo} \label{expansions}
Let $\M = \langle M, <, +, \cdot, \dots \rangle$ be a polynomially
bounded o-minimal expansion of a real closed field, $v$ its natural
valuation, $G$ its value group, $k$ its residue field. We assume
that its prime model $\mathcal P$ is archimedean. Then $\M$ is
$\A_{\alpha}$-saturated if and only if
\begin{enumerate}
\item $G$ is $\A_{\alpha}$-saturated as an
ordered $Q-$vector space
\item $k \cong \R$
\item for every substructure $\mathcal M^{\prime}$ with
$\dim (\mathcal M^{\prime}/\mathcal P)<\A_{\alpha}$, every pseudo Cauchy sequence in $M^{\prime}$ has a pseudo limit in $M$.
\end{enumerate}
\end{theo}
\begin{proof}
The proof follows the lines of the proof of \cite[Theorem 6.2]
{kkmz}, and we shall only point out the main adjustments which
concern the cardinality of the expanded language, and bounding the
dimension of definable closures. We begin by observing that the
proof of necessity is analogous to that of \cite[pp.88-89]{kkmz};
one needs to take $Q$-linear spans instead of $\Q$-linear spans. Now
assume conditions (1), (2) and (3) and we show that $\M$ is
$\A_{\alpha}$-saturated. Let $q$ be a complete $1$-type over $\M$
with parameters in $A \subset M$, with $|A|<\A_{\alpha}$.
Let $\M^{''}$ be an elementary extension of $\M$ in which $q(x)$ is
realized, and $x_0 \in M^{''}$ such that $\M^{''} \models q(x_0)$.
As noted earlier to realize $q$ in $\M$ it is necessary and
sufficient to realize the cut that $x_0$ makes in $\M^{\prime}:=
\dcl(A)\subseteq \M$
$$
q^{\prime}(x) := \{b \leq x \, ; \, b \in M^{\prime}, \, q \vdash b
\leq x \} \cup \ \{x \leq c \, ; \ c \in M^{\prime}, \, q \vdash x
\leq c \}.
$$
 As we will see in realizing the
cut $q^{\prime}$ instead of the type $q$ some case distinction will
be needed according to whether $\aleph_{\alpha} > \ell$ (so
$|\dcl(A)| < \aleph_{\alpha}$) or $\aleph_{\alpha} \leq \ell$ (so
$|\dcl(A)| \geq \aleph_{\alpha}$). As in the proof of \cite[Theorem
6.2]{kkmz} we set $B := \{b \in M^{\prime} \, ; \,  q \vdash b < x\}
\mbox{ and } C := \{c \in M^{\prime} \, ; \,  q \vdash x < c\}$. Consider
 $\Delta = \{v(d - x_0) \, | \, d \in M^{\prime} \}$ and
the three possible cases:
\begin{enumerate}
\item[$(a)$] {\em Immediate transcendental case}: $\Delta$ has no largest element.
\item[$(b)$] { \em Value transcendental case}: $\Delta$ has a largest element $\gamma \not \in v(M^{\prime})$.
\item[$(c)$] { \em Residue transcendental case}: $\Delta$ has a largest element $\gamma \in v(M^{\prime})$.
\end{enumerate}
We deal with case (a)  as in \cite[p. 86]{kkmz}, taking into account
that we get a pseudo-Cauchy sequence in $M^{\prime}$ and that $\dim
(\mathcal M / \mathcal P) \leq |A|< \aleph_{\alpha}$. \noindent In
case $(b)$ we need to deal with the type
$$t(y) = \{v(c - d_0) < y ; c \in C \} \cup
\{y < v(b - d_0); b \in B, b > d_0\}
$$
over $G$ with parameters in $v(M^{\prime})$ as in \cite[p.
87]{kkmz}.
 Set $G^{\prime}=v(M^{\prime})$. If  $\A_{\alpha}>\ell$ then $|(G^{\prime}|<\A_{\alpha}$ and by hypothesis $(1)$ we can realize $t(y)$ in $G$.
 The delicate case is when $\A_{\alpha}\leq\ell$. By Remark
 \ref{tconvx2}, $\dim_Q(v(G^{\prime})\leq \dim (\mathcal M^{\prime}/\mathcal P) \leq |A| <\aleph_{\alpha}\>.$
 Fixing a $Q$-basis of $G^{\prime}$ of cardinality $< \aleph_{\alpha}
 $, we can rewrite $t(y)$ as a type $t^{\prime}(y)$ with parameters now from the basis. Since $G$ is
$\A_{\alpha}$-saturated,  we can realize $t^{\prime}(y)$ in $G$. The
remaining of the argument, as well as the argument for case $(c)$
are as in \cite[pp. 87-88]{kkmz}.
\end{proof}
We note that the characterization provides a construction method for
$\A_{\alpha}$-saturated polynomially bounded o-minimal expansions of
real closed fields, using (the maximally valued) fields of power
series.
\begin{ex}
{\rm Consider the uncountable language $\mathcal L_{an}$ of ordered
fields expanded with symbols for restricted analytic functions.
Construct the power series field  $\R((G))$ (see \cite{book}) where
$G$ is the $\ell$-saturated group constructed at the end of previous
section. The field $\R((G))$ can be endowed with a polynomially
bounded o-minimal $\mathcal L_{an}$-structure (see
\cite{speissegger}). By Theorem \ref{expansions} it is
$\ell$-saturated.}
\end{ex}
\begin{rem} \label{napm}
(1) Assume that $\M$ is polynomially bounded but the prime model is
non archimedean. Consider the coarsening $w$ of $v$ whose valuation
ring is the convex hull of the prime model in $M$. Working with this
valuation, one gets an analogue of Theorem \ref{expansions}. \\ (2)
Assume now that $\M$ is power bounded, but not polynomially bounded,
i.e. the field of exponents $Q$ is non-archimedean. In this case,
working with the induced valuation $w_Q$ on the ordered $Q$-vector
space $w_Q(K)$, one gets  analogues of Theorems \ref{doag} and \ref{expansions}.
\end{rem}


\vs

\end{document}